\documentclass[11pt]{amsart}
\usepackage{amsfonts,amssymb,amsmath,amsthm}
\usepackage{url}
\usepackage{enumerate}
\usepackage[all]{xy}

\usepackage{amsmath,amsfonts,amsthm,url,color,amssymb}
\usepackage{graphicx}
\usepackage{algpseudocode, algorithm}
\usepackage{hyperref}
\usepackage{todonotes}

\usepackage[norefs,nocites]{refcheck}

\newcommand{\G}{\Gamma}
\newcommand{\F}{\mathbb F}
\newcommand{\C}{\mathbb C}

\newtheorem{theorem}{Theorem}[section]
\newtheorem{lemma}[theorem]{Lemma}

\newtheorem{proposition}[theorem]{Proposition}
\theoremstyle{definition}
\newtheorem{definition}[theorem]{Definition}

\newtheorem{problem}[theorem]{Problem}
\theoremstyle{remark}
\usepackage{amssymb}
\newtheorem{remark}[theorem]{Remark}
\usepackage{enumerate}

\author{Jean Godard}
\address{Departamento de Matem\'{a}tica, Universidade Federal de Minas Gerais, UFMG, Belo Horizonte, MG, 31270-901, Brazil}
\email{jgodard@ufmg.br}

\author{Lucas Reis}
\address{Departamento de Matem\'{a}tica, Universidade Federal de Minas Gerais, UFMG, Belo Horizonte, MG, 31270-901, Brazil}
\curraddr{}
\email{lucasreismat@gmail.com}

\begin{document}

\title[Graphs from quadratic forms and vector spaces]{Graphs from quadratic forms and vector spaces over finite fields}

\maketitle

\begin{abstract}
Let $q$ be an odd prime power, let $n\ge 2$, and let $V\subsetneq \F_{q^n}$ be a proper $\F_q$-vector subspace. Given a nonzero quadratic form $Q(X,Y)\in \F_{q^n}[X,Y]$, we consider the graph $\G(Q,V)$ that naturally arises from the condition $Q(X,Y)\in V$. We determine all quadratic forms $Q$ for which $\G(Q,V)$ is undirected for every $V$. Besides the case $Q(X,Y)=XY$, studied earlier by the second author, this essentially leads to the forms $X^2\pm Y^2$ and the family $Q_b(X, Y):=X^2+bXY+Y^2, b\ne 0$. We then study connectedness and clique number for the corresponding graphs. Our results reveal a clear contrast between these cases. The graphs $\G(X^2\pm Y^2, V)$ are well structured, disconnected and their clique number can be as large as $\# V$. On the other hand, the family $Q_b$ seems to yield less structured graphs: the graphs are connected (in fact, of diameter $2$) if $\# V\ge q^{3n/4}$ and, in many cases, their clique number is $o(\# V)$. Our proofs are mainly based on character sums, while requiring a few algebraic and combinatorial ideas. We end the paper with some  open problems and remarks, including a short discussion of the complementary case where $q$ is even.
\end{abstract}


\section{Introduction} 
\label{intro}
Graphs defined from algebraic conditions have been extensively studied in several contexts. For instance, Cayley graphs~\cite{cayley1878} play a fundamental role in the interplay between group theory and combinatorics, being responsible for many extremal constructions in graph theory. For more information on Cayley graphs, see~\cite{babai1979, godsil2001, zhou2022}. We also have the zero-divisor graphs in ring theory, where these graphs reflect the multiplicative structure through the relation $XY=0$: see~\cite{akbari2006, anderson1999, anderson2011, beck1988} for more details.  

Let $q$ be a prime power and let $\F_q$ be the finite field with $q$ elements. In the finite field setting, a classical algebraic graph construction is the Paley graph, in which there is a directed edge from $a\in \F_q$ to $b\in \F_q$ if $a-b$ is a nonzero square of $\F_q$. Most of the literature is focused in the case $q\equiv 1\pmod 4$, for which the corresponding graph is undirected. More generally, one may replace squares by $d$-th powers, where $d$ is an integer satisfying $q\equiv 1\pmod {2d}$. Graphs of this kind have been studied by several authors, and questions on the clique number and relations with Waring's problem over finite fields have been discussed in ~\cite{cohen,paley, pod, csaba,yip2}. More recently, other graph contructions from finite fields have been explored: see~\cite{kim2026,mans2020, yip2025}.

The Paley-like graph constructions are based on the interaction between differences in $\F_q$ and multiplicative subgroups of $\F_q^*$. More recently, in~\cite{reis2} the author discusses the multiplicative-additive analogue of the latter. Specifically, given an additive subgroup $V$ of $\F_q$ (i.e., a vector subspace), one can consider the graphs $\Gamma(XY^{-1}, V)$ with vertex set $\F_q^*$ and directed edges $a\to b$ if $a\ne b$ and $ab^{-1}\in V$. However, if we want the graph to be undirected, this means that $V\setminus \{0\}$ is closed under the inversion map $z\mapsto z^{-1}$. According to~\cite{bence}, up to a family of exceptions, the latter implies that $V$ is a subfield. When $V$ is a subfield, the graph $\Gamma(XY^{-1}, V)$ has a very simple structure, hence not interesting to explore. In particular, in~\cite{reis2} the author replaces the condition $XY^{-1}\in V$ by $XY\in V$, making the corresponding graph undirected regardless the vector subspace $V$. The paper is mainly focused on the study of cliques of this graph. In particular, a nice description of maximal cliques is given, along with bounds for the clique number of the graph. What reveals is that the clique number for many $V$'s can be bounded nontrivially using a beautiful sum-product estimate from Bourgain-Glibichuck~\cite{bourgain}. Moreover, the problem of computing the largest possible clique number obtained from a proper subspace $V\subsetneq \F_q$ translates into a combinatorial problem involving the vanishing of certain bilinear forms~\cite{ah}. 

Motivated by the $XY$ construction, in this paper we explore a more general setting over fields of odd characteristic, replacing the expression $XY$ by a generic quadratic form. 

\begin{definition}
    Given a nonzero quadratic form $Q(X, Y)=aX^2+bXY+cY^2\in \F_{q^n}[X, Y]$ and an $\F_q$-vector subspace $V\subseteq \F_{q^n}$, let $\G(Q, V)$ be the directed graph with vertex set $\F_{q^n}$ and directed edges $x\to y$ if $x\ne y$ and $Q(x, y)\in V$. 
\end{definition}
We say that $\G(Q, V)$  is {\em undirected} if, for every $x, y\in \F_{q^n}$, both or none of the directed edges $\{x\to y, y\to x\}$ are in $\G(Q, V)$. If $\G(Q, V)$ is undirected, the edges are denoted by $(x, y)$ and $\omega(Q, V)$ is the clique number of $\G(Q, V)$. Here we are mostly interested in discussing the clique number and conectedness of such graphs. In particular, we will only consider quadratic forms $Q$ for which $\G(Q, V)$ is always undirected, regardless the subspace $V$. Our first result provides a complete description of such $Q$'s.

\begin{theorem}\label{thm:first}
    Let $q$ be an odd prime power and let $Q(X, Y)=aX^2+bXY+cY^2\in \F_{q^n}[X, Y]$ be a nonzero quadratic form. Then $\G(Q, V)$ is undirected for every $\F_q$-vector subspace $V\subseteq \F_{q^n}$ if and only if $a=c$ or $b=0$ and $a=-c\ne 0$. 
\end{theorem}

We observe that if $\lambda \in \F_{q^n}^*$, then $\Gamma(Q, V)\cong \Gamma(\lambda Q, \lambda V)$, where $\lambda V$ is again an $\F_q$-vector subspace. In particular, it suffices to explore the graphs by considering the quadratic forms up to a multiplication by a nonzero scalar. Combining this observation with Theorem~\ref{thm:first}, we end up with four types of quadratic forms:
\begin{itemize}
\item $Q_*(X, Y):=XY$;
    \item $Q_+(X, Y):=X^2+Y^2$;
    \item $Q_-(X, Y):=X^2-Y^2$;
    \item $Q_b(X, Y):=X^2+bXY+Y^2$ with $b\in \F_{q^n}^*$. 
\end{itemize}
The quadratic form $Q_*(X, Y)$ was earlier explored in~\cite{reis2}; in this case the graph is always connected as there is an edge between $0$ and any other element of $\F_{q^n}$. So we concentrate our study on the remaining three quadratic forms above. 
For these quadratic forms, we observe that for each $\alpha\in \F_{q^n}$ and $v\in V$, the equation $Q(\alpha, y)=v$ has at most $2$ solutions. In particular, the degree of every vertex in $\G(Q, V)$ is at most $2\# V$, yielding the trivial bound 
$$\omega(Q, V)\le 2\# V.$$ 
The case $V=\F_{q^n}$ yields the complete graph on $q^n$ vertices, regardless the quadratic form $Q$ so we only consider proper vector subspaces.
Our main result provides information on the connectedness and clique number of the graphs $\G(Q, V)$ and it can be read as follows.

\begin{theorem}\label{thm:main}
Let $q$ be an odd prime power and let $V\subsetneq \F_{q^n}$ be a proper $\F_q$-vector subspace, where $n\ge 2$. Moreover, for each $u\in \F_{q^n}$, let $N(u, V)$ be the number of elements $z\in \F_{q^n}$ such that $z^2\in u+V:=\{u+v:v\in V\}$. Then the following hold:

\begin{enumerate}[(i)]
\item For $Q=Q_+, Q_-$, the graph $\Gamma(Q, V)$ is disconnected. Moreover, if $\# V>q^{n/2}$, then $\Gamma(Q_-, V)$ has exactly $\frac{q^n}{\# V}$ connected components.

\item $\omega(Q_+, V)=2$ if $N(0, V)=1$. Otherwise,  $$\omega(Q_+, V)=N(0, V).$$
In particular, if $n$ is odd, then $\omega(Q_+, V)=\# V$. 
    
    \item $\omega(Q_-, V)= \max\{N(u^2, V): u\in \F_{q^n}\}$. In particular, if $n$ is odd, then $\omega(Q_-, V)\ge N(0, V)=\# V$.

\item For every $b\in \F_{q^n}^*$, $\omega(Q_b, V)\le q^{n/2}+2$.
\end{enumerate}  
In particular, for $Q=Q_{\pm}$, we have 
\begin{equation}\label{eq:near}
    |\omega(Q, V)-\# V|\le q^{n/2}.
\end{equation}
\end{theorem}


 Theorem~\ref{thm:main} is obtained through a combination of ideas and tools. We prove that the graphs $\Gamma(R, V)$ with $R=Q_+, Q_-$ are well structured and, in particular, we provide a complete characterization of its connected components and maximal cliques. The latter yields a formula for the clique number by means of the number of squares in $V$ and its translates.  
    For the quadratic forms $Q_b$, the corresponding graphs seem to be less structured and we estimate their clique number via a character sum bound due to Gyamarti and Sárközy~\cite{GS}.

\begin{remark}\label{rem:optimal}
    The uniform bound in item (iv) in Theorem~\ref{thm:main} is essentially sharp. In fact, for $n=2k$ and  $U=\F_{q^{k}}$, the bound implies $\omega(Q_b, U)\le q^k+2$. On the other hand, as $U$ is a field, it is closed under addition and multiplication. Hence if
    $b\in \F_{q^k}$, then $Q(x, y)\in U$ for every $x, y\in \F_{q^k}$ and so the vertices from $U$ induce a clique in $\G(Q_b, U)$. In particular,    $\omega(Q_b, U)\ge q^k$. 
\end{remark}

     In constrast to the cases $Q=Q_+, Q_-$, Theorem~\ref{thm:main} entails that the clique number of $\G(Q_b, V)$ cannot be as large as $\# V$ when $V$ has large dimension. In fact, if $V$ has dimension larger than $n/2$, Eq.~\eqref{eq:near}  implies that $\frac{\omega(Q_b, V)}{\# V}=O(q^{-1/2})$. On the other hand, while the graphs $\G(Q_{\pm}, V)$ are always disconnected, in many cases the graph $\G(Q_b, V)$ is connected. In fact, we prove a much stronger result.

     \begin{theorem}\label{thm:diam}
         Let $q$ be an odd prime, let $n\ge 2$ be a positive integer and let $V\subsetneq \F_{q^n}$ be a proper $\F_q$-vector space. If  $\# V\ge q^{3n/4}$, then for every $b\in \F_{q^n}^*$ the graph $\G(Q_b, V)$ has diameter $2$. 
     \end{theorem}

The proof of Theorem~\ref{thm:diam} is mainly based on the {\em character sum method}, a classical tool in Number Theory and Finite Fields for detecting special patterns with the help of exponential sum estimates. In particular, we express the indicator function for vector spaces in finite fields by means of additive characters and, with the help of Weil's bound, we prove that every two vertices of $\G(Q_b, V)$ are connected to a third vertex if $\# V\ge q^{3n/4}$. Surprisingly enough, if we employ the same method to prove something weaker like $\Gamma(Q_b, V)$ having diameter $\le s$ with $s>2$, we would essentially need $\# V$ to be much larger than $q^{3n/4}$. This suggests that our approach alone is not sufficient to break the barrier $3n/4$. On the other hand, in the following remark we show that $n/2$ is the natural threshold for our problem.  

\begin{remark}\label{rem:disc}
    Although Theorem~\ref{thm:diam} might not be sharp, we can find disconnected graphs emerging from high dimensional vector spaces. For instance, set $n=2k$ and let $\alpha$ be a nonsquare of $\F_{q^n}$. Since every element of $\F_{q^k}$ is a square, every nonzero element of $U=\alpha\F_{q^k}$ is a nonsquare. In particular if $(0, z)$ is an edge of $\Gamma(Q_b, U)$, we obtain $0\ne z^2=Q_b(0, z)\in U$, a contradiction. In particular, $0$ is an isolated vertex of the graph.
\end{remark}

The paper is organized as follows. In Section 2 we provide some technical machinery, including bounds for certain character sums. In Section 3 we provide the main properties of the graphs arising from the quadratic forms $Q_{\pm}$. In Section 4 we prove our results and in Section 5 we provide remarks and problems that arise from our research, including a short discussion on these graphs for fields of characteristic $2$.

\section{Preparation}
In this section we provide some technical results that are essential to prove our main theorems.

\begin{lemma}\label{lem:square}
Let $q$ be an odd prime power and let $n\ge 2$ be odd. if $V\subseteq \F_{q^n}$ is an $\F_q$-vector space, then there exist $\# V$ elements $y\in \F_{q^n}$ such that $y^2\in V$. 
\end{lemma}
\begin{proof}
Recall that $\F_q^*$ is cyclic, hence there exists $\theta\in \F_q^*$ of multiplicative order $q-1$. In particular, $\theta^{(q-1)/2}=-1$. Since $q, n$ are odd, the same holds for $\frac{q^n-1}{q-1}$ and so 
$$\theta^{\frac{q^n-1}{2}}=(\theta^\frac{q-1}{2})^{\frac{q^n-1}{q-1}}=-1.$$
Therefore, $\theta$ is not a square in $\F_{q^n}$. Since $\theta\in \F_q$ and $V$ is an $\F_q$-vector subspace, we see that the map $z\mapsto \theta z$ fixes $0$ and provides a bijection between the set of nonzero squares and nonsquares of $V$. In particular, $V$ has exactly $\frac{\#V-1}{2}$ nonzero squares, each yielding $2$ elements $y$ with $y^2\in V$ and the result follows by adding the element $0=0^2\in V$. 
\end{proof}

\subsection{Characters and character sums}
We briefly recall some standard definitions and facts on characters. We skip proofs and  refer to Chapter 5 of~\cite{LN} for more details on characters over finite fields. If $G$ is a finite abelian group, a character of $G$ is a group homomorphism $\chi:G\to \C^*$, where $\C$ is the field of complex numbers. Since every element of $G$ has finite order, $\chi(g)$ is a complex root of unity for every $g\in G$, and so $|\chi(g)|=1$. The trivial character is the constant character $\chi(g)=1$ for every $g\in G$.

An additive character of $\F_q$ is a character of the additive group $(\F_q,+)$. It is known that, for $q=p^n$ with $p$ prime, the additive characters of $\F_q$ comprise the maps $\psi_a, a\in \F_q$ given by $$\psi_a(x)=\exp\left(\frac{2\pi i \mathrm{Tr}(ax)}{p}\right),$$ where $\mathrm{Tr}(x)=x^{p^{n-1}}+\cdots+x^p+x$ denotes the absolute trace map from $\F_q$ to $\F_p$. In particular, $\psi_0$ is the trivial character. The character $\psi_1(x)$ is called the {\em canonical} additive character of $\F_q$. In particular, $\psi_1(x)$ is not trivial and satisfies 
$\psi_a(x)=\psi_1(ax)$ for every $a, x\in \F_{q}$.

A multiplicative character of $\F_q$ is a simply a character of the multiplicative cyclic group $\F_q^*$. Throughout the paper, every multiplicative character $\chi$ is extended to $0\in \F_q$ by setting $\chi(0)=0\in \mathbb C$. In particular, the trivial multiplicative character $\chi_0$ satisfies $\chi_0(x)=1$ for $x\in \F_q^*$ and $\chi_0(0)=0$. When $q$ is odd, we denote by $\eta$ the quadratic character of $\F_q$, that is, $\eta(x)=1$ if $x$ is a nonzero square, $\eta(x)=-1$ if $x$ is a nonsquare, and $\eta(0)=0$.

\begin{remark}\label{rem:rest}
 If $V\subsetneq \F_q$ is an additive subgroup (i.e., a proper vector subspace), then there exists a nontrivial additive character $\psi_V$ of $\F_q$ such that $\psi_V(v)=1$ for every $v\in V$.  
\end{remark}

The following lemma provides a simple expression for the indicator function of vector subspaces by means of additive characters. Its proof is a straightforward application of the {\em Poisson Summation Formula} for  finite abelian groups (see Theorem 1 of Chapter 12 in~\cite{terras}).

\begin{lemma}\label{lem:indicator}
Let $V\subseteq \F_{q^n}$ be an $\F_q$-vector space of size $q^j$. Then there exists an $\F_q$-vector space $V_*\subseteq \F_{q^n}$ of size $q^{n-j}$ such that, for every $x\in \F_{q^n}$, we have
$$\sum_{u\in V_*}\psi_u(x)=\begin{cases}q^{n-j}&\text{if}\quad x\in V,\\ 0&\text{if}\quad x\not\in V.\end{cases}$$    
\end{lemma}

The next result provides a uniform bound for multiplicative character sums over affine subspaces.

\begin{lemma}\label{lem:boundaffine}
    If $V\subseteq \F_{q^n}$ is an $\F_q$-vector space and $\chi$ is a nontrivial multiplicative character of $\F_{q^n}$, then for every $y\in \F_{q^n}$ we have  
    $$\left |\sum_{v\in V}\chi(y+v)\right|\le q^{n/2}.$$
    In particular, if $q$ is odd and $\# V>q^{n/2}$, then $y+V$ always contains a nonzero square.
\end{lemma}

\begin{proof}
The first statement is a straightforward application of Corollary~2 in~\cite{GS} by taking $X=y+V$ and $Y=V$. For the second statement, if $y+V$ does not contain a nonzero square, it contains at least $\# V-1$ nonsquares. If $\eta$ denotes de quadratic character we obtain 
$\sum_{v\in V}\eta(y+v)\le 1-\# V$, hence 
$\# V-1\le q^{n/2}$. However, since $\# V>q^{n/2}$ and $V$ is an $\F_q$-vector subspace, we have $\# V\ge q^{n/2+1/2}$. The latter yields $q^{n/2}+1\ge q^{n/2+1/2}$, a contradiction with $q\ge 3$.
\end{proof}

The following lemma is a straightforward application of Theorem 4 in~\cite{GS}.

\begin{lemma}\label{lem:GS}
    Let $q$ be an odd prime power, fix $Q(X, Y)=X^2+bXY+Y^2\in \F_{q^n}[X, Y]$ with $b\ne 0$ and let $\psi$ be a nontrivial additive character of $\F_{q^n}$. Then for every nonempty sets $A, B\subseteq \F_{q^n}$, we have 
    $$\left|\sum_{a\in A, b\in B}\psi(Q(a, b))\right|\le (q^n\#A\#B)^{1/2}.$$
\end{lemma}

Finally, we also need Weil's bound for additive character sums.

\begin{lemma}\label{lem:weil}
    Let $f\in \F_q[X]$ be a polynomial of degree $d\ge 1$ such that $\gcd(d, q)=1$. Then for every nontrivial additive character $\psi$ of $\F_q$ we have
$$\left|\sum_{x\in \F_q}\chi(f(x))\right|\le (d-1)q^{1/2}.$$
    
\end{lemma}

\section{The graphs $\G(Q_{\pm}, V)$}
In this section we provide the structure of the graphs arising from the quadratic forms $Q=Q_{\pm}$. For the quadratic form $Q_+$, we have the following result.

\begin{proposition}\label{prop:Q+}
    Let $V\subsetneq \F_{q^n}$ be a proper $\F_q$-vector subspace and let $S_V$ be the set of squares in $V$. If $S_V=\{0\}$, then $\omega(Q_+, V)=2$ and $\G(Q_+, V)$ is triangle-free. Otherwise, the set $C_V:=\{u\in \F_{q^n}: u^2\in S_V\}$ induces a maximal clique in $\Gamma(Q_+, V)$ of size at least $3$, and every other maximal clique of this graph has size $2$. In particular, $\Gamma(Q_+, V)$ is disconnected.
\end{proposition}
\begin{proof}
For the first statement, suppose that $\omega(Q_+, V)\ge 3$ and let $x_1, x_2, x_3\in \F_{q^n}$ be elements inducing a triangle in $\G(Q_+, V)$ with $x_1\ne 0$. Hence $x_1^2+x_2^2, x_1^2+x_3^2, x_2^2+x_3^2\in V$. Since $V$ is a vector subspace, it follows that
$$2x_1^2=(x_1^2+x_2^2)+(x_1^2+x_3^2)-(x_2^2+x_3^2)\in V.$$
Since $q$ is odd, we obtain $x_1^2\in V$ and then $S_V\ne \{0\}$.
For the second statement, suppose $S_V\ne \{0\}$ and let $u^2\in S_V$ with $u\ne 0$. It is direct to verify that $C_V$ induces a clique containing the vertices $u, -u$ and $0$. On the other hand, from the proof of the first statement, the vertices of every clique of size at least $3$ must lie in $C_V$. This completes the proof of the second statement. 

For the last statement, observe that if $(x, v)$ is an edge of $\G(Q_+, V)$
 with $v\in C_V$, then $x^2+v^2\in V$. Since $v^2\in V$, we obtain $x^2\in V$ and so $x\in C_V$. Hence $C_V$ induces a maximal connected component of $\G(Q_+, V)$. It is clear that $C_V$ has size at most $2\# V$. Therefore, if $\G(Q_+, V)$ is connected, then $2\# V\ge q^n$ and, since $q$ is odd, the latter is equivalent to $\# V=q^n$, a contradiction with $V\ne \F_{q^n}$.
\end{proof}

In the following lemma we provide a simple description of the connected components of $\Gamma(Q_-, V)$.

\begin{lemma}\label{lem:Q-}
Let $V\subseteq \F_{q^n}$ be a proper $\F_q$-vector subspace  and, for each $y\in \F_{q^n}$, set 
$$A_y:=\{u\in \F_{q^n}: u^2\in y^2+V\}.$$ 
Then each $A_y$ induces a maximal connected component of $\Gamma(Q_-, V)$, which is also a (maximal) clique. In particular, $\Gamma(Q_-, V)$ is disconnected.
\end{lemma}

\begin{proof}
For $x, y\in \F_{q^n}$, write $x\sim y$ if $x^2-y^2\in V$. We clearly have $x\sim x$ and, if $x\sim y$, then $y\sim x$. Moreover, if $x\sim y$ and $y\sim z$, then 
$x^2-y^2, y^2-z^2\in V$. Since $V$ is a subspace, we obtain
$$x^2-z^2=(x^2-y^2)+(y^2-z^2)\in V,$$
hence $x\sim z$. In other words, $\sim$ is an equivalence relation. Moreover, the equivalence classes are precisely the sets $A_y, y\in \F_{q^n}$. 
Therefore, each connected component of $\Gamma(Q_-, V)$ is a clique comprising the subgraph induced by $A_y$ for some $y$. In particular, each connected component has size at most $\# A_y\le 2\# V$. As in the proof of Proposition~\ref{prop:Q+}, we conclude that $\Gamma(Q_-, V)$ is disconnected.  
\end{proof}

\begin{remark}\label{rem:Q-}
 We observe that, for $y, z\in \F_{q^n}$, we have $A_y=A_z$ if and only if $y^2+V=z^2+V$. As there are exactly $\frac{q^n}{\# V}$ distinct affine subspaces of the form $c+V$, we conclude that $\Gamma(Q_-, V)$ has at most $\frac{q^n}{\# V}$ connected components with equality if and only if every affine subspace $c+V$ is of the form $y^2+V$. The latter holds if and only if $c+V$ contains a square for every $c\in \F_{q^n}$: thanks to Lemma~\ref{lem:boundaffine}, this is true when $\# V> q^{n/2}$.
\end{remark}

\section{Proof of the results}
Here we provide the proofs of Theorems~\ref{thm:first} and~\ref{thm:main}. We divide them into subsections.

\subsection{Proof of Theorem~\ref{thm:first}}
It is direct to verify that the quadratic forms $Q$ listed in Theorem~\ref{thm:first} are such that $\Gamma(Q, V)$ is undirected for every subspace $V$. Conversely, assume that the latter holds and fix $u\in\F_{q^n}$ with $u\ne 0$. For each $v\in \F_{q^n}$ with $v\ne u$, let $U=Q(u, v)\F_q$ be the vector subspace generated by $Q(u, v)$. It is clear that $u\to v$ is a directed edge in $\Gamma(Q, U)$, hence $v\to u$ is also a directed edge in this graph, that is, $Q(v, u)\in U$. In other words, there exists $\lambda_v\in \F_{q}$ such that $Q(u, v)=\lambda_v Q(v, u)$. Since there are $q^{n}-1$ choices of $v$ and $q$ choices of $\lambda_v$, by the Pigeonhole Principle, there exists $\lambda_* \in \F_{q}$ such that the equation
$$Q(u, Y)=\lambda_* Q(Y, u),$$
has at least $\left\lceil\frac{q^n-1}{q}\right\rceil=q^{n-1}$ solutions. Write $Q(X, Y)=aX^2+bXY+cY^2$, hence 
$$Q(u, Y)-\lambda_* Q(Y, u)=(c-\lambda_* a)Y^2+bu(1-\lambda_*)Y+u^2(a-\lambda_* c).$$
The latter is a polynomial in $\F_q[Y]$ of degree at most $2$, hence it has at most $2$ solutions unless it vanishes.
Since $q$ is odd and $n\ge 2$, we obtain $q^{n-1}\ge 3>2$ and so $Q(u, Y)-\lambda_*Q(Y, u)=0$. Since $u\ne 0$, we obtain the following linear system of equations
$$\begin{cases}
    c&=\lambda_* a\\ b&=\lambda_*b\\  a&=\lambda_* c.
\end{cases}$$

The case $\lambda_*=0$ yields $a=b=c=0$, a contradiction. If $\lambda_*=1$, we obtain $a=c$. If $\lambda_*\ne 0, 1$, we obtain $b=0$ and $0\ne c=\lambda^2 c$, hence $\lambda_*=-1$. In this case, $(a, b, c)=(-c, 0, c)$ with $c\ne 0$. This completes the proof. 

\subsection{Proof of Theorem~\ref{thm:main}}
We split the proof into parts. 

\begin{enumerate}[(i)]
    \item Item (i) follows directly by Proposition~\ref{prop:Q+}, Lemma~\ref{lem:Q-} and Remark~\ref{rem:Q-}. 

\item For item (ii), Proposition~\ref{prop:Q+} shows that $\omega(Q_+, V)=2$ if $V$ does not contain a square distinct from $0$ and, otherwise, $\omega(Q_+, V)$ equals the number of $y\in \F_{q^n}$ such that $y^2\in V$ which is, by definition, $N(0, V)$.
Moreover, if $n$ is odd, Lemma~\ref{lem:square} entails that $N(0, V)=\# V>1$, completing the proof.

\item Similarly to item (ii), we prove item (iii) (by employing Lemma~\ref{lem:Q-} in this case).

\item For item (iv), fix $b\in \F_{q^n}^*$, set $M=\omega(Q_b, V)$ and let $S\subseteq \F_{q^n}$ be a set with $M$ vertices inducing a clique in $\G(Q_b, V)$. In particular, $Q(s, t)\in V$ for every $s, t\in S$ with $s\ne t$. As $V\ne \F_{q^n}$, Remark~\ref{rem:rest} guarantees the existence of a nontrivial additive character $\psi_V$ of $\F_{q^n}$ such that $\psi_V(v)=1$ for every $v\in V$. Therefore, $\psi_V(Q_b(s, t))=1$ for every $s, t\in S$ with $s\ne t$. Applying Triangle's Inequality and Lemma~\ref{lem:GS} with $A=B=S$, we obtain 
\begin{align*}M(M-1)=\left|\sum_{s, t\in S\atop s\ne t}\psi_{V}(Q_b(s, t))\right| &\le \left|\sum_{s, t\in S}\psi_{V}(Q_b(s, t))\right|+\sum_{s\in S}|\psi_{V}(Q_b(s, s))|\\ &\le (q^nM^2)^{1/2}+M.\end{align*}
Thus $M\le q^{n/2}+2$, completing the proof.

\item For the last statement, we first treat the pairs $(Q_+, V)$ where $V$ does not contain squares apart from $0$. In this case, item (ii) entails that $\omega(Q_+, V)=2$. However, thanks to Lemma~\ref{lem:boundaffine}, we have $\# V\le q^{n/2}$ and the inequality $|\omega(Q_+, V)-\# V|\le q^{n/2}$ trivially holds. For the remaining cases, items (ii) and (iii) entail that, for $Q=Q_{\pm}$, there exists $u\in \F_{q^n}$ such that $\omega(Q, V)$ equals the number of elements $y\in \F_{q^n}$ with $y^2\in u+V$. Now observe that, if $\eta$ denotes the quadratic character of $\F_{q^n}$, then for every $a\in \F_{q^n}$ the equation $x^2=a$ has $\eta(a)+1$ solutions in $\F_{q^n}$. Hence 
$$\omega(Q, V)=\sum_{v\in V}(1+\eta(u+v))=\# V+\sigma$$
where $\sigma=\sum_{v\in V}\eta(v+u)$. From Lemma~\ref{lem:boundaffine}, $|\sigma|\le q^{n/2}$ and the result follows.
\end{enumerate}

\subsection{Proof of Theorem~\ref{thm:diam}}
Since $V\ne \F_{q^n}$, it is clear that $\Gamma(Q_b, V)$ is not the complete graph and so its diameter is at least $2$. Set $\# V=q^{n-j}$ with $j\le n/4$ and let $V_*$ be as in Lemma~\ref{lem:indicator}. Suppose, by contradiction, that $\G(Q_b, V)$ has diameter greater than $2$. In particular, there exists $x, y\in \F_{q^n}$ with $x\ne y$ such that, for every $z\in \F_{q^n}$, at least one of the elements in the set $\{Q_b(x, z), Q_b(y, z)\}$ is not in $V$. From Lemma~\ref{lem:indicator} we obtain
$$0=\Delta:=\sum_{z\in \F_{q^n}}\left(\sum_{u\in V_*}\psi_u(Q_b(x, z))\right)\left(\sum_{v\in V_*}\psi_v(Q_b(y, z))\right).$$
If, for each pair $(u, v)$ with $u, v\in V_*$, we write \begin{align*}\sigma_{u, v}&=\sum_{z\in \F_{q^n}}\psi_{u}(Q_b(x, z))\psi_v(Q_b(y, z))\\&= \sum_{z\in \F_{q^n}}\psi_1((u+v)z^2+bz(ux+vy)+ux^2+vy^2),\end{align*}
we obtain 
$\Delta=\sum_{u, v\in V_*}\sigma_{u, v}$. We clearly have $\sigma_{0, 0}=q^n$. Since $b\ne 0$, the polynomial in $z$ appearing in the expression for $\sigma_{u, v}$ has degree $d_{u, v}\in \{1, 2\}$ if $(u, v)\ne (0, 0)$. Since $q$ is odd, we obtain $\gcd(d_{u, v}, q)=1$ and so we can apply Weil's bound. Lemma~\ref{lem:weil} implies that $|\sigma_{u, v}|\le q^{n/2}$ if $(u, v)\ne (0, 0)$. From Triangle's Inequality and the equality $\# V_*=q^j$, we obtain
$$0=|\Delta|\ge |\sigma_{0, 0}|-\sum_{u, v\in V_*\atop (u, v)\ne (0, 0)}|\sigma_{u, v}|\ge q^n-(q^{2j}-1)q^{n/2}.$$
Therefore, $q^{2j}>q^{n/2}$, a contradiction with $j\le n/4$. This concludes the proof.

\section{Conclusions and open problems}
In this paper we discussed the connectedness and clique number of graphs arising from vector spaces $V$ over finite fields through the quadratic forms $X^2\pm Y^2$ and $Q_b(X, Y)=Y^2+bXY+Y^2, b\ne 0$. Our study revealed a contrast between these quadratic forms regarding both connectedness and clique number. Of course there could be many other interesting properties to explore in these graphs.

\begin{problem}
    Discuss other relevant combinatorial issues on the graphs $\Gamma(Q, V)$.  
\end{problem}

Theorem~\ref{thm:main} generally provides a closed formula for the clique number of $\Gamma(X^2\pm Y^2, V)$ via the number of squares in $V$ and its translates. The latter essentially requires the computation of the character sum $\left|\sum_{v\in V}\eta(u+v)\right|$ along the affine subspaces $u+V$, where $\eta$ is the quadratic character. In particular, the estimate in Eq.~\eqref{eq:near} follows from the classical fact that this character sum does not exceed $q^{n/2}$. However, the bound becomes trivial in the range $\# V<q^{n/2}$. In~\cite{reis1} the author breaks the $n/2$-barrier under reasonably generic choices of $V$ and $n$, hence some nontrivial bounds for the clique number $\omega(X^2\pm Y^2, V)$ in the range $\# V<q^{n/2}$ can be derived. Nevertheless, improving estimates for character sums along affine subspaces is by itself a hard open problem.

Moving to the quadratic forms $Q_b$, Theorem~\ref{thm:main} provides the uniform bound $\omega(Q_b, V)\le q^{n/2}+2$, which does not improve the trivial bound $\le 2\# V$ if $\# V<q^{n/2}$. On the other hand, following Remark~\ref{rem:optimal}, if $V$ is a subfield of $\F_{q^n}$ and $b\in V$, we easily obtain the lower bound $\omega(Q_b, V)\ge \# V$. We propose the following problem.

\begin{problem}
Let $V\subseteq \F_{q^n}$ be a proper $\F_q$-vector space such that $\# V<q^{n/2}$ and, for every $\alpha\in \F_{q^n}$,  $\alpha V$ is not a subfield of $\F_{q^n}$.
\begin{enumerate}
    \item Provide nontrivial estimates for $\omega(Q_b, V)$.
    \item Prove or disprove: $\omega(Q_b, V)=o(\# V)$ as $q\to +\infty$.
\end{enumerate}
\end{problem}
Finally, in light of Theorem~\ref{thm:diam}, we also propose further discussion on connectedness of the graphs $\G(Q_b, V)$.

\begin{problem}
For each integer $n\ge 4$, find the smallest integer $s(n)$ with the following property: if $q$ is large and $\# V\ge q^{s(n)}$, then $\Gamma(Q_b, V)$ is connected for every $b\in \F_{q^n}^*$.
\end{problem}

Theorem~\ref{thm:diam} and Remark~\ref{rem:disc} provide the bounds $s(n)\le 3n/4$ and, for $n$ even, $s(n)>n/2$. In particular, we obtain $s(4)=3$, $s(6)\in \{4, 5\}$, $s(8)\in \{5, 6\}$ and $s(10)\in \{6, 7\}$. 
\subsection{Even characteristic}
We end the paper with some comments on the study of the graphs $\G(Q, V)$ in fields of characteristic $2$. Theorem~\ref{thm:first} also holds if $q$ is even. In~\cite{reis2} the authors also considers the graphs emerging from $Q(X, Y)=XY$ when $q$ is even, and it is evident its contrast to the case where $q$ is odd. Now, for $Q(X, Y)=X^2\pm Y^2=X^2+Y^2=(X+Y)^2$, we see that if $\tilde(V)=\{u\in \F_{q^n}: u^2\in V\}$, then $\tilde{V}$ is a vector subspace of the same size of $V$ and $\Gamma(Q, V)$ coincides with the graph induced by the relation $x+y\in \tilde{V}$. The latter is quite trivial: the graph comprises $\frac{q^n}{\# V}$ connected components, each being a clique of size $\# V$ representing the cosets $u+\tilde{V}$ of the quotient $\F_{q^n}/\tilde{V}$ of additive groups. Finally, our results for $Q_b$ no longer hold in characteristic $2$. In fact, the conditions in Lemma~\ref{lem:weil} and Theorem~4 in~\cite{GS} (which implies Lemma~\ref{lem:GS}) can be easily violated for suitable choices of $Q_b$ if $q$ is even. 

\begin{problem}
    Study the graphs $\G(Q_b, V)$ for fields of even characteristic and compare the findings with the results from the odd characteristic case.
\end{problem}

\section*{Acknowledgements}
 L. Reis was supported by CNPq Grant 310583/2025-0.


\end{document}